\documentclass[twoside]{amsart}

\usepackage[colorlinks=true, citecolor=blue]{hyperref}
\usepackage{amsmath, amsthm, amssymb}
\usepackage{epsfig, graphicx}
\usepackage{verbatim}

\newtheorem{theorem}{Theorem}
\newtheorem{proposition}[theorem]{Proposition}
\newtheorem{corollary}[theorem]{Corollary}
\newtheorem{lemma}[theorem]{Lemma}

\begin{document}

\newcommand{\T}		{\mathbb{T}}
\newcommand{\Z}		{\mathbb{Z}}
\newcommand{\C}		{\mathbb{C}}
\newcommand{\N}		{\mathbb{N}}
\newcommand{\D}		{\mathbb{D}}

\newcommand{\A}	{\mathcal{A}}
\newcommand{\Ha}	{\mathcal{H}}
\newcommand{\U}	{\mathcal{S}}
\newcommand{\Po}	{\mathcal{P}}
\newcommand{\I}		{\mathcal{I}}
\newcommand{\B}	{\mathcal{B}}

\newcommand{\w}	{\textnormal{w}}
\newcommand{\z}		{\textnormal{z}}
\newcommand{\rank}	{\textnormal{rank}}

\title[Meromorphic Extendibility and Rigidity of Interpolation]{Meromorphic Extendibility and Rigidity of Interpolation}

\author[M.~Raghupathi]{Mrinal Raghupathi} \address{Department of Mathematics, Vanderbilt University, Nashville, TN, 37240}

\email{mrinal.raghupathi@vanderbilt.edu}

\author[M.~Yattselev]{Maxim L. Yattselev}

\address{Corresponding author, Department of Mathematics, University of Oregon, Eugene, OR, 97403}

\email{maximy@uoregon.edu}

\thanks{The research was conducted during M.Y.'s stay at Vanderbilt University as a Visiting Scholar.}

\begin{abstract}
  Let $\T$ be the unit circle, $f$ be an $\alpha$-H\"older continuous
  function on $\T$, $\alpha>1/2$, and $\A$ be the algebra of
  continuous function in the closed unit disk $\overline\D$ that are
  holomorphic in $\D$. Then $f$ extends to a meromorphic function in
  $\D$ with at most $m$ poles if and only if the winding number of
  $f+h$ on $\T$ is bigger or equal to $-m$ for any $h\in\A$ such that
  $f+h\neq0$ on $\T$.
\end{abstract}



\maketitle

\section{Main Results}

Let $g$ be a non-vanishing continuous function on a simple Jordan
curve $T$. Denote by $\w_T(g)$ the winding number of $g(T)$ around the
origin. That is, $2\pi\w_T(g)$ is equal to the change of the argument
of $g$ on $T$ when the curve $T$ is traversed in the positive
direction with respect to $D$, the interior domain of $T$. Denote by
$\A(D)$ the algebra of functions continuous on $\overline D$ and
holomorphic in $D$. 

Motivated by the work of Alexander and Wermer \cite{AlWer00} and Stout
\cite{Stout00}, Globevnik \cite{Glob_CADSII05} obtained the following
characterization of functions in the disk algebra $\A:=\A(\D)$, where
$\D$ is the unit disk. 
\begin{theorem}[Globevnik~\cite{Glob_CADSII05}]
A continuous function $f$ on the unit
  circle $\T$ extends holomorphically through $\D$ if and only if
  $\w_\T(f+q)\geq0$ for each polynomial $q$ such that $f+q\neq0$ on
  $\T$.
\end{theorem} 
A shorter proof, based on the notion of badly-approximable
functions, was obtained by Khavinson~\cite{Khav_CADSII05}.

The polynomials are a dense subalgebra of $\A$. Thus, for any $h\in\A$
such that $f+h\neq0$ on $\T$, there exists a polynomial $q$ satisfying
$|h-q|<|f+h|$ on $\T$. Then
\begin{equation}
\label{eq:2wn}
\w_\T(f+q) = \w_\T(f+h+q-h) = \w_\T(f+h)+\w_\T\left(1+\frac{q-h}{f+h}\right) = \w_\T(f+h).
\end{equation}
Hence, the hypothesis of Globevnik's result above could equivalently
have been stated as $w_\T(f+q)\geq 0$ for all $q\in \A$.

In later work Globevnik~\cite{Glob04} was able to generalize the above
result to multiply-connected domains. Namely, let $D$ be a domain
whose boundary $T$ consists of finitely many pairwise disjoint Jordan
curves. The winding number of a continuous, non-vanishing function $g$
on $T$, $\w_T(g)$, is defined as the sum of the individual winding
numbers on each Jordan curve constituting $T$ oriented positively with
respect to $D$.  Then \emph{a continuous function $f$ on $T$ is the
  trace of a function in $\A(D)$ if and only if $\w_T(f+h)\geq0$ for
  any $h\in\A(D)$ such that $f+h\neq0$ on $T$.}

The next natural question is to characterize functions that admit
\emph{meromorphic} continuation. On this path Globevnik \cite{Glob08b}
showed the following: \emph{A continuous function $f$ on $\T$ extends
  meromorphically through $\D$ with at most $m\in\Z_+:=\N\cup\{0\}$
  poles\footnote{The counting of poles and zeros is done including
    multiplicities.} in $\D$ if and only if $\w_\T(pf+q)\geq-m$ for
  each pair of polynomials $p$ and $q$ such that $pf+q\neq0$ on $\T$.}
Moreover, for sufficiently smooth $f$ it could be assumed that
$p\in\Po_m$, i.e., $p$ is a polynomial of degree at most $m$. In fact,
this assumption can be made regardless of the smoothness of
$f$. Indeed, it follows from the same method used in
\cite{Khav_CADSII05} for the case $m=0$. One merely replaces the best
holomorphic approximant by the best meromorphic approximant with at
most $m$ poles. It is known from AAK-theory \cite{AAK71} that the error
of such approximation has constant modulus on $\T$ and winding number
at most $-(2m+1)$. Finally, in \cite{Glob08}, Globevnik derived a
similar result on meromorphic extendibility for multiply connected
domains where $p$ and $q$ belong to $\A(D)$.

The work on holomorphic extendibility in \cite{Glob08b} and
\cite{Glob08} indicates that the above results could be improved. In
particular, is it the case that the conclusion of the above results on
meromorphic extensibility are still true, under the weaker hypothesis
that $\w_T(f+q)\geq -m$ for all $q\in \A(D)$. In this work we give an
affirmative answer to this question in the case of the unit circle
when $f$ is sufficiently smooth. We now state our main result.

\begin{theorem}
\label{thm:merext}
Let $f$ be an $\alpha$-H\"older continuous function on $\T$,
$\alpha>1/2$. Let $m\in\Z_+$. Then $f$ extends to a meromorphic
function with at most $m$ poles in $\D$ if and only if
\begin{equation}
\label{eq:merext}
\w_\T(f+h) \geq -m
\end{equation}
for every $h\in\A$ such that $f+h\neq0$ on $\T$.
\end{theorem}

As indicated in the argument following~\eqref{eq:2wn}, this theorem can be
 stated equivalently with $h$ in the algebra of analytic polynomials.

Notice that the necessity of \eqref{eq:merext} is trivial. Indeed, if
$f=g/q$, $g\in\A$ and $q\in\Po_m$, then
$\w_\T(f+h)=\w_\T(g+qh)-\w_\T(q)\geq-m$ as the winding number of
$g+qh$ is non-negative and $\w_\T(q)$ is equal to the number of zeros
of $q$ in $\D$. Thus, we need only to show that \eqref{eq:merext} is
sufficient for $f$ to be the trace of a function meromorphic in $\D$
with at most $m$ poles there.

Let $\varphi\in\A$ be bi-Lipschitz in $\overline\D$. That is, there
exists a finite positive constant $c$ such that
\[
(1/c)|z_1-z_2|\leq|\varphi(z_1)-\varphi(z_2)| \leq c|z_1-z_2|, \quad
z_1,z_2\in\overline\D.
\]
Put $D_\varphi:=\varphi(\D)$ and $L_\varphi:=\varphi(\T)$. Clearly, it
holds that $h\in\A(D_\varphi)$ if and only if
$h\circ\varphi\in\A$. Moreover, it is true that
$\w_{L_\varphi}(g)=\w_\T(g\circ\varphi)$ for any continuous function
$g$ on $L_\varphi$. It is also true that $\varphi$ preserves H\"older
classes. Thus, the following result is another immediate consequence
of Theorem~\ref{thm:merext}.

\begin{corollary}
  Let $\varphi\in\A$ be bi-Lipschitz in $\overline\D$. Then
  Theorem~\ref{thm:merext} remains valid when $\T$ and $\D$ are
  replaced by $L_\varphi$ and $D_\varphi$, respectively.
\end{corollary}

In what follows, we suppose that $m\in\N$ since the case $m=0$ was
shown in \cite{Glob04}. Moreover, without loss of generality we may
assume that $f\notin\A$. This, in particular, implies that
$f_-\not\equiv0$, where
\begin{equation}
\label{eq:anti}
f_-(z) := \frac{1}{2\pi i}\int_\T\frac{f(t)dt}{z-t}, \quad z\notin\overline\D,
\end{equation}
is the anti-analytic part of $f$. It is known that each $f_-$ is a
function holomorphic outside of $\overline\D$ whose trace on $\T$ is
also $\alpha$-H\"older continuous.

Our approach lies in recasting the condition on the winding number of
$f+h$, $h\in\A$, as a certain rigidity property of interpolation of
$f$ by polynomials. For a function $g$ holomorphic in $\D$ we denote
by $\z(g)$ the number of zeros of $g$ in $\D$ counting
multiplicities. The following proposition is central to our approach.

\begin{proposition}
\label{prop:interpol}
Let $f$ be an $\alpha$-H\"older continuous function on $\T$,
$\alpha>1/2$. Let $m\in\N$. Then \eqref{eq:merext} is satisfied if and
only if
\begin{equation}
\label{eq:zeros}
\z(f_n+p)\leq m+n
\end{equation}
holds for for any $n\in\Z_+$ and any polynomial $p\in\Po_n$, where
$f_n(z) = z^nf_-(1/z)$, $z\in\D$.
\end{proposition}

The condition on H\"older continuity of $f$ appearing in the statement
of Theorem~\ref{thm:merext} is needed exactly for the proof of this
proposition. It ensures the fact that the image of $\T$ under $f_n+p$
has no interior points.

Motivated by the interpolation rigidity property \eqref{eq:zeros}, we
define the following classes
\[
\I_{n,m} := \left\{g:~g~\mbox{is holomorphic in}~\D,~\z(g+p)\leq
  m+n~\mbox{for any}~p\in\Po_n\right\},
\]
$n\in\Z_+$, $m\in\N$. That is, a function $g$ holomorphic in $\D$
belongs to the class $\I_{n,m}$ if and only if any polynomial of
degree at most $n$ interpolates $g$ at no more than $n+m$ points. It
follows immediately from the definition that
$\I_{n,m}\subset\I_{n,m+1}$. Moreover, $g\in\I_{0,1}$ if and only if
$g$ is a univalent function in $\D$ and
$g\in\I_{0,m+1}\setminus\I_{0,m}$ if and only if $g$ is an
$(m+1)$-valent function.

Observe that a function $f$, continuous on $\T$, is the trace of a
function meromorphic in $\D$ if and only if $f_-$ is a rational
function. Thus, the sufficiency part of Theorem~\ref{thm:merext} is a
consequence of Proposition~\ref{prop:interpol} and the following
theorem applied with $g(\cdot)=f_-(1/\cdot)$.

\begin{theorem}
\label{thm:overint}
Let $g$ be a holomorphic function in $\D$ such that $g_n\in\I_{n,m}$
for any $n\in\Z_+$, $g_n(z)=z^ng(z)$. Then $g$ is a rational function
of type $(m,m)$ holomorphic in $\D$.
\end{theorem}

Clearly if $g$ is a rational function of type $(m,m)$, then the
numerator of $g_n+p$, $p\in\Po_n$, belongs to $\Po_{n+m}$ and
therefore $\z(g_n+p)\leq n+m$ for any $n\in\Z_+$.

Theorem~\ref{thm:overint} was initially proved for the case $m=1$
independently by Ruscheweyh \cite{Rusch76} and Kirjackis
\cite{Kir76}. In this work, we elaborate on the approach devised in
\cite{Kir76}.

\section{Proofs}

\begin{proof}[Proof of Proposition~\ref{prop:interpol}]
  We prove this proposition in two simple steps. First, we show that
  \eqref{eq:zeros} is equivalent to
\begin{equation}
\label{eq:windzer}
\z_{\{|z|>1\}}(f_-+q) \leq \deg(q)+m
\end{equation}
for any algebraic polynomial $q$, where $\z_{\{|z|>1\}}(f_-+q)$ stands
for the number of zeros of the function $f_-+q$ in $\{|z|>1\}$.
Second, we establish the equivalence of \eqref{eq:windzer} and
\eqref{eq:merext}.

Let $q\in\Po_n\setminus\Po_{n-1}$ be given. Set
$p(z):=z^nq(1/z)$. Then $p\in\Po_n$ and $p(0)\neq0$. Since
\begin{equation}
\label{eq:aux}
z^n(f_-+q)(1/z) = z^nf_-(1/z)+z^nq(1/z) = (f_n+p)(z),
\end{equation}
it holds that $\z_{\{|z|>1\}}(f_-+q) = \z(f_n+p)$. That is,
\eqref{eq:zeros} implies \eqref{eq:windzer}. Conversely, let
$p\in\Po_n$ be given. Denote by $k$ the multiplicity of the zero of
$p$ at the origin, $k=0$ if $p(0)\neq0$. Set $q(z):=z^np(1/z)$. Then
$q\in\Po_{n-k}$ and $\z_{\{|z|>1\}}(f_-+q) = \z(f_n+p)-k$ by
\eqref{eq:aux}. Hence, \eqref{eq:windzer} implies
\eqref{eq:zeros}. Thus, \eqref{eq:windzer} and \eqref{eq:zeros} are
equivalent.

Let $h\in\A$ be such that $f+h\neq0$ on $\T$. Since $f_++h\in\A$,
$f_+:=f-f_-$, there exists a polynomial $q$ such that
$|q-f_+-h|<|f+h|$ on $\T$. Then we get as in \eqref{eq:2wn} that
$\w_\T(f+h)=\w_\T(f_-+q)$. As $f_-+q$ is a meromorphic function in
$\{|z|>1\}$, $\w_\T(f_-+q)$ is simply the difference between the
number of poles and the number of zeros of $f_-+q$ there. That is,
\begin{equation}
\label{eq:zerosandwinding}
\w_\T(f_-+q) = \deg(q) - \z_{\{|z|>1\}}(f_-+q)
\end{equation}
since the point at infinity is the only possible pole of $f_-+q$ and
the order of this pole is $\deg(q)$. Thus, we get that $\w_\T(f+h) =
\deg(q) - \z_{\{|z|>1\}}(f_-+q)$. That is, \eqref{eq:windzer} implies
\eqref{eq:merext}.

Let $q$ be a polynomial and assume that $f_-+q\neq0$ on $\T$. Applying
\eqref{eq:merext} with $h=q-f_+$, we get that
$\w_\T(f_-+q)\geq-m$. The desired conclusion then follows from
\eqref{eq:zerosandwinding}. Assume now that $f_-+q=0$ for some points
on $\T$. As $f_-$ is $\alpha$-H\"older continuous function on $\T$
with $\alpha>1/2$, $(f_-+q)(\T)$ has no interior points
\cite{SalZyg45}. Thus, for any $\delta>0$ there exists a complex
number $\epsilon$ such that $|\epsilon|=\delta$ and
$f_-+q+\epsilon\neq0$ on $\T$. Then \eqref{eq:windzer} follows from
the Rouch\'e's theorem. This finishes the proof of the proposition.
\end{proof}

Let $g$ be an $m$-valent function in $\D$. That is, there exists a
constant $a$ such that $a$ interpolates $g$ at $m$ points in $\D$, say
$ \{z_1,\ldots,z_m\}$, in the Hermite sense. Put
\begin{equation}
\label{eq:transf}
g^a(z) := \frac{cz^m(g(z)-a)}{(z-z_1)\cdots(z-z_m)},
\end{equation}
where the constant $c$ is chosen so $z^{-m}g^a(z)\to1$ as
$z\to0$. Clearly, $g^a$ is again a holomorphic function in $\D$.

Denote by $\U_m$ the set of $m$-valent holomorphic functions in $\D$
satisfying $z^{-m}g(z)\to1$ as $z\to0$, and set
\[
\B_m := \left\{g\in\U_m:~g(z)=\frac{z^m}{d(z)},~d\in\Po_m\right\}.
\]
Necessarily, it holds that $d(0)=1$ and $d(z)\neq0$ in $\D$. Recall
\cite[Thm. 5.3]{Hayman} that for any $a\in D_m:=\{|z|<1/4^m\}$, the
equation $g(z)=a$ has exactly $m$ roots in $\D$ for any $g\in\U_m$.

For functions in $\U_m$ we adopt the notation
\[
g(z)=z^m+(g)_1z^{m+1}+\cdots+(g)_kz^{m+k}+\cdots.
\]

\begin{lemma}
\label{lem:ell}
There exists a set of polynomials of $m$ variables, say
$\{\ell_{m,k}\}$, $k\in\N$, $k>m$, such that $g\in\B_m$ if and only if
$(g)_k=\ell_{m,k}((g)_1,\ldots,(g)_m)$.
\end{lemma}
\begin{proof}
Let $g\in\B_m$. That is, $g(z)=z^m/d(z)$, $d(z)=1+d_1z+\cdots+d_mz^m$. Then
\[
1 \equiv (1+d_1z+\cdots+d_mz^m)\left(1+(g)_1z+\cdots+(g)_mz^m+\cdots\right).
\]
Hence,
\begin{equation}
\label{eq:rel1}
\left\{
\begin{array}{llll}
d_1 &=& -(g)_1, & {} \\
d_k &=& -(g)_k-(g)_{k-1}d_1-\cdots-(g)_1d_{k-1}, & k\in\{2,\ldots,m\},
\end{array}
\right.
\end{equation}
and
\begin{equation}
\label{eq:rel2}
(g)_n = -(g)_{n-1}d_1-\cdots-(g)_{n-m}d_m, \quad n\in\N, \quad n>m.
\end{equation}
Clearly, relations \eqref{eq:rel1} allow us to express $d_k$
polynomially through $(g)_1,\ldots,(g)_k$ for each
$k\in\{1,\ldots,m\}$. Polynomials $\ell_{m,n}$ are constructed then
using \eqref{eq:rel2} inductively in $n$ by plugging in the
corresponding expressions for $d_k$.

To prove the ``if'' part, observe that relations \eqref{eq:rel1}
uniquely determine the set of coefficients $d_1,\ldots,d_m$ for a
given set $\{(g)_1,\ldots,(g)_m\}$. Taking into account the way
polynomials $\ell_{n,m}$ were constructed, we see that relations
\eqref{eq:rel2} take place with these $d_1,\ldots,d_m$. That is,
$g(z)/z^m=1/d(z)$ for some polynomial $d\in\Po_m$.
\end{proof}

\begin{lemma}
\label{lem:d}
Let $g\in\U_m$. Then $|(g^a)_k|=|(g)_k|$, $k\in\{1,\ldots,m\}$, if and only if $g\in\B_m$.
\end{lemma}
\begin{proof}
  It is a trivial computation to verify that $g^a=g$ for any
  $g\in\B_m$. Thus, we only need to prove the ``only if'' part.

  Let $g\in\U_m$ and suppose that $|(g^a)_k|=|(g)_k|$,
  $k\in\{1,\ldots,m\}$. Transformation \eqref{eq:transf} can be
  equivalently written as
\begin{equation}
\label{eq:us0}
g^a(z) = \frac{z^m}{d(z)}\left(1-\frac{g(z)}{a}\right), \quad d(z)=\left(1-\frac{z}{\zeta}\right)\left(1-\frac{z}{\zeta_2}\right)\cdots\left(1-\frac{z}{\zeta_m}\right),
\end{equation}
where $g(\zeta)=g(\zeta_j)=a$, $j\in\{2,\ldots,m\}$. Clearly,
$\zeta=\zeta(a)$ and $\zeta_j=\zeta_j(a)$, $j\in\{2,\ldots,m\}$, are,
in fact, holomorphic functions of $a\in D_m^*:=D_m\setminus[0,1/4^m]$
(or any other domain obtained from $D_m$ by removing a Jordan arc that
connects the origin and some point on the boundary of this
disk). Write
\[
\frac{1}{d(z)} = 1 + \left(\frac1d\right)_1z+\cdots+\left(\frac1d\right)_mz^m + \cdots.
\]
As in \eqref{eq:rel1}, it holds that
\begin{equation}
\label{eq:us1}
\left(\frac1d\right)_k = -s_k\left(-\frac1\zeta,-\frac{1}{\zeta_2},\ldots,-\frac{1}{\zeta_m}\right)-\sum_{j=1}^{k-1}\left(\frac1d\right)_{k-j}s_j\left(-\frac1\zeta,-\frac{1}{\zeta_2},\ldots,-\frac{1}{\zeta_m}\right),
\end{equation}
where $s_k(b_1,\ldots,b_N)$ is the $k$-th symmetric function of $b_1,\ldots,b_N$, i.e.,
\begin{equation}
\label{eq:symfun}
\left\{
\begin{array}{lll}
s_N(b_1,\ldots,b_N) &=& \prod_{j=1}^N b_j \smallskip \\
s_{N-1}(b_1,\ldots,b_N) &=& \sum_{i=1}^N\prod_{j\neq i} b_j \smallskip \\
{} &\cdots& {} \smallskip \\
s_1(b_1,\ldots,b_N) &=& \sum_{j=1}^Nb_j.
\end{array}
\right.
\end{equation}
It is a simple computation to check using \eqref{eq:us0} that
\begin{equation}
\label{eq:us2}
(g^a)_j = \left(\frac1d\right)_j, \quad j\in\{1,\ldots,m-1\}, \quad \mbox{and}  \quad (g^a)_m = \left(\frac1d\right)_m-\frac1a.
\end{equation}
In particular, this means that each $(g^a)_k$, $k\in\{1,\ldots,m\}$,
is a holomorphic function of $a$ in $D_m^*$. Thus, $(g^a)_k\equiv c_k$
for some constants $c_k$ independently of $a$ since by the conditions
of the lemma functions $(g^a)_k$ have constant modulus. Hence,
\eqref{eq:us2} and \eqref{eq:us1} yield that there exist constants
$c_k^*$, $k\in\{1,\ldots,m\}$, such that
\begin{equation}
\label{eq:us3}
s_k\left(-\frac1\zeta,-\frac{1}{\zeta_2},\ldots,-\frac{1}{\zeta_m}\right) \equiv c_k^*, \quad j\in\{1,\ldots,m-1\},
\end{equation}
and
\begin{equation}
\label{eq:us4}
s_m\left(-\frac1\zeta,-\frac{1}{\zeta_2},\ldots,-\frac{1}{\zeta_m}\right) \equiv c_m^* -\frac1a
\end{equation}
for all $a\in D_m$. It is easy to verify that
\begin{equation}
\label{eq:us5}
\left\{
\begin{array}{ll}
s_1\left(-\frac1\zeta,-\frac{1}{\zeta_2},\ldots,-\frac{1}{\zeta_m}\right) & = -\frac1\zeta +s_1\left(-\frac{1}{\zeta_2},\ldots,-\frac{1}{\zeta_m}\right) \smallskip \\
s_k\left(-\frac1\zeta,-\frac{1}{\zeta_2},\ldots,-\frac{1}{\zeta_m}\right) & = -\frac1\zeta s_{k-1}\left(-\frac{1}{\zeta_2},\ldots,-\frac{1}{\zeta_m}\right)+s_k\left(-\frac{1}{\zeta_2},\ldots,-\frac{1}{\zeta_m}\right) \smallskip \\
s_m\left(-\frac1\zeta,-\frac{1}{\zeta_2},\ldots,-\frac{1}{\zeta_m}\right) & = -\frac1\zeta s_{m-1}\left(-\frac{1}{\zeta_2},\ldots,-\frac{1}{\zeta_m}\right)
\end{array}\right.
\end{equation}
for $k\in\{2,\ldots\,m-1\}$. Combining \eqref{eq:us5} with \eqref{eq:us3}, we derive that
\begin{equation}
\label{eq:us6}
s_m\left(-\frac1\zeta,-\frac{1}{\zeta_2},\ldots,-\frac{1}{\zeta_m}\right) = -\left(\frac{1}{\zeta^m}+\frac{c_1^*}{\zeta^{m-1}}+\cdots+\frac{c_{m-1}^*}{\zeta}\right).
\end{equation}
Finally, since $a=g(\zeta)$, we get from \eqref{eq:us4} and \eqref{eq:us6} that
\[
g(\zeta) = \frac{\zeta^m}{c_m^*\zeta^m+\cdots+c_1^*\zeta+1}.
\]
That is, $g\in\B_m$.
\end{proof}

\begin{lemma}
\label{lem:subh}
Let $\ell$ be a polynomial of $N$ variables and $g\in\U_m$. Then $|\ell((g^a)_{k_1},\ldots,(g^a)_{k_N})|$ is a subharmonic function of $a\in D_m$.
\end{lemma}
\begin{proof}
  As the modulus of a holomorphic function is subharmonic, we only
  need to show that $(g^a)_k$ is a holomorphic function of $a$ for any
  $k\in\N$.

  Let $a\in D_m$ and $z_1,\ldots,z_m$ be such that $g(z_j)=a$. Then
  each $z_j(a)$ is a holomorphic function of $a\in
  D_m^*:=D_m\setminus[0,1/4^m]$ and for each $j_1$ there exists
  $j_2(\neq j_1)$ such that $z_{j_1}^+\equiv z_{j_2}^-$ on
  $[0,1/4^m]$, where $z_j^+$ and $z_j^-$ are the traces of $z_j$ from
  above and below on $[0,1/4^m]$. This, in particular, means that
  $s_j(z_1(a),\ldots,z_m(a))$ is a holomorphic function of $a\in D_m$
  by the principle of analytic continuation for each
  $j\in\{1,\ldots,m\}$, where $s_j$ is the $j$-th symmetric function
  \eqref{eq:symfun}. Thus,
\[
d(z,a)=(z-z_1(a))\cdots(z-z_m(a))
\]
is a holomorphic function of $a\in D_m$ for each $z$ and it vanishes
as a function of $a$ only if $g(z)=a$. Hence, $g^a(z)$ is a
holomorphic function of $a\in D_m$ for each $z$. Since,
\[
(g^a)_k = \frac{1}{\rho^{2(k+m)}}\int_{|\tau|=\rho} \bar\tau^{k+m}g^a(\tau)\frac{d\tau}{2\pi\rho},
\]
the conclusion of the lemma follows.
\end{proof}

In the proof of Lemma~\ref{lem:d} we pointed out that $\B_m$ is
invariant under \eqref{eq:transf}. In fact, there are no larger subset
of $\U_m$ with this property.

\begin{lemma}
\label{lem:bm}
$\B_m$ is the largest subset of $\U_m$ invariant under \eqref{eq:transf}.
\end{lemma}
\begin{proof}
  Denote by $\B$ the largest subset of $\U_m$ invariant under
  \eqref{eq:transf}. Observe that $\B$ is compact with respect to the
  locally uniform convergence in $\D$. Indeed, it is known
  \cite[Thm. 5.3]{Hayman} that $\U_m$ is a normal family and therefore
  $\B$ is bounded. Let $\{g_\alpha\}\subset\B$ be a convergent
  sequence and $g$ be the limit function. Clearly, $g_\alpha^a$
  converge to $g^a$ locally uniformly in $\D$ as well. Then
  $g,g^a\in\U_m$ by Hurwitz's theorem. By iterating this process, we
  indeed see that $g\in\B$.

Fix $n\in\N$, $n>m$, and denote by $\B^n$ the subset of $\B$ consisting of functions maximizing the functional
\[
\phi_n(g) := |(g)_n-\ell_{m,n}((g)_1,\ldots,(g)_m)|,
\]
where the polynomial $\ell_{m,n}$ was introduced in
Lemma~\ref{lem:ell}. Since $\B$ is compact, $\B^n$ is non-empty and
clearly compact. Let $g\in\B^n$. By Lemma~\ref{lem:subh},
$\phi_n(g^a)$ is a subharmonic function of $a$. As $g^0=g$, the
maximum principle for subharmonic functions yields that $g^a\in\B^n$
for all $a\in D_m$.

Among all the functions $g$ in $\B^n$, chose those that have maximal
modulus of $(g)_m$. Using the subharmonicity of $|(g)_m|$ and
compactness of $\B^n$ as in the previous paragraph, we deduce that
this set is compact, non-empty, and invariant under
\eqref{eq:transf}. Further, among functions $g$ in the latter set,
choose those that maximize $|(g)_{m-1}|$. Once more, it follows that
this new set is compact, non-empty, and invariant under
\eqref{eq:transf}. Repeat this procedure for
$(g)_{m-2},\ldots,(g)_1$. After the last step, we reach a non-empty
set invariant under \eqref{eq:transf} with the property
$|(g)_k|=|(g^a)_k|$ for all $k\in\{1,\ldots,m\}$, $a\in D_m$, and each
$g$ in this set. By Lemma~\ref{lem:d}, this set is contained in
$\B_m$.

We have shown that some of the maximizing functions for the functional
$\phi_n$ belong to $\B_m$. However, $\phi_n\equiv0$ on $\B_m$ by the
very definition of $\ell_{n,m}$ and therefore $\phi_n\equiv0$ on
$\B$. Since this is true for any $n$, Lemma~\ref{lem:ell} yields that
$\B=\B_m$.
\end{proof}

\begin{proof}[Proof of Theorem~\ref{thm:overint}]
Let $g$ be such that $g_n\in\I_{n,m}$ for all $n\in\Z_+$. Without loss of generality we may assume that $m$ is the smallest natural number with this property. Thus, there exists the smallest integer in $\Z_+$, say $n_0$, such that
\[
\z(g_{n_0}+q) = m+n_0 \quad \mbox{for some} \quad q\in\Po_{n_0}.
\]
Set $k:=m+n_0$ and $v$, $\deg(v)=k$, to be the monic polynomial
vanishing at the zeros of $g_{n_0}+q$. Define
\begin{equation}
\label{eq:p2}
y(z) := cz^k\frac{(g_{n_0}+q)(z)}{v(z)},
\end{equation}
where $c$ is a normalizing constant such that $z^{-m}y(z)\to1$ as
$z\to0$. Then $y$ is a holomorphic function in $\D$ and for any
$p\in\Po_n$ it holds that
\begin{eqnarray}
\z(y_n+p) &=& \z\left(\frac{cg_{k+n+n_0}+cz^{k+n}q+pv}{v}\right) = \z\left(g_{k+n+n_0}+z^{k+n}q+\frac1cpv\right) - k \nonumber \\
\label{eq:p3}
{} &\leq&  m+k+n+n_0 - k=k+n
\end{eqnarray}
as $(cz^{k+n}q+pv)\in\Po_{k+n+n_0}$. That is, $y\in\U_k$ and
$y_n\in\I_{n,k}$ for all $n\in\Z_+$. Since \eqref{eq:p2} can be viewed
as transformation \eqref{eq:transf} applied to $g_{n_0}+q$ with $a=0$,
estimates analogous to \eqref{eq:p3} show that $y^a\in\U_k$ and
$y^a_n\in\I_{n,k}$, $n\in\Z_+$, for all $a\in D_k$. Applying
transformation \eqref{eq:transf} to $y^a$, we again get that the newly
obtained function belongs to $\U_k$ and its shifts by $z^n$ belong to
$\I_{n,k}$. Clearly, this process can be continued indefinitely. That
is, $y$ belongs to the subset of $\U_k$ invariant under
\eqref{eq:transf}, i.e, $y\in\B_k$ by Lemma~\ref{lem:bm}. Hence,
$y(z)=z^k/d(z)$, $d\in\Po_k$. Then
\[
g(z) = \frac{1}{z^{n_0}}\left(\frac1c\frac{v(z)}{d(z)}-q(z)\right).
\]
Since $\deg(v)=k$ and $\deg(q)=n_0$, $g$ is a rational function of
type $(\max\{m,l\},l)$ with $l:=\deg(d)\leq m+n_0$. Thus, it remains
to show that $l\leq m$.

Set $\D_\rho:=\{z:~|z|<\rho\}$ for fixed $\rho<1$. Clearly,
$g\in\A(\D_\rho)$ and it holds that
\begin{equation}
\label{eq:p4}
\z_\rho(g_n+p) \leq n+m
\end{equation}
for any $p\in\Po_n$ and $n\in\Z_+$, where $\z_\rho(g_n+p)$ is the
number of zeros $g_n+p$ in $\D_\rho$. Since polynomials are contained
and dense in $\A(\D_{1/\rho})$, we get as in
Proposition~\ref{prop:interpol} that \eqref{eq:p4} is equivalent to
\begin{equation}
\label{eq:p5}
\w_{\T_{1/\rho}}(g(1/\cdot)+h)\geq-m
\end{equation}
for all $h\in\A(\D_{1/\rho})$ such that $g(1/\cdot)+h\neq0$ on
$\T_{1/\rho}:=\{z:~|z|=1/\rho\}$. It is easy to see that
$g(1/\cdot)=s/r$, where $s\in\Po_l$ and
$r\in\Po_l\setminus\Po_{l-1}$. As $r(z)=z^ld(1/z)$, all the zeros of
$r$, say $\{w_1,\ldots,w_l\}$, belong to
$\overline\D\subset\D_{1/\rho}$. Fix a determination of each $\log
s(w_j)$, $j\in\{1,\ldots,l\}$, and let $u$ be a polynomial
interpolating the values $\log s(w_j)$ at the points $w_j$,
respectively. Set $h:=(e^u-s)/r$. Then $h\in\A(\D_{1/\rho})$ and
\[
\w_{\T_{1/\rho}}(f+h) = \w_{\T_{1/\rho}}\left(\frac sr+\frac{e^u-s}{r}\right) = \w_{\T_{1/\rho}}\left(\frac{e^u}{r}\right) = -l
\]
since $e^u\neq0$ in $\overline\D_{1/\rho}$. Thus, $l\leq m$ by \eqref{eq:p5}.
\end{proof}

\end{document}